\documentclass[12pt]{amsart}
\usepackage{amssymb}
\usepackage{amsmath}
\usepackage{bbm}
\usepackage{color}
\usepackage{verbatim}

\usepackage{amsthm}

\definecolor{darkred}{rgb}{0.75,0,0.3}

\newcommand\ac{\tilde a}
\newcommand\AND{\quad\text{and}\quad}
\newcommand\C{\mathbb C}
\newcommand\DD{\mathbb D}
\newcommand\ep{\varepsilon}

\newcommand\Fc{\wt F}
\newcommand\ff{g}
\newcommand\gh{g} 
\newcommand\hor{\mathfrak{h}}

\newcommand\im{\mathfrak{i}\,}
\newcommand\KK{\mathcal{K}}
\newcommand\la{\lambda}
\newcommand\Lap{\Delta}
\newcommand\mm{\mathsf m}

\newcommand\N{\mathbb N}

\newcommand\res{\mathsf{res}}

\newcommand\uno{\mathbf{1}}
\newcommand\wh{\widehat}
\newcommand\wt{\widetilde}

\begin{document}
\theoremstyle{theorem}

\newtheorem{theorem}{Theorem}
\numberwithin{theorem}{section}
\newtheorem{pro}[theorem]{Proposition}
\newtheorem{lem}[theorem]{Lemma}
\newtheorem{cor}[theorem]{Corollary}

\theoremstyle{definition}
\newtheorem{dfn}[theorem]{Definition}
\newtheorem{rem}[theorem]{Remark}

$\,$ \vspace{-1truecm}
\title{
Boundary behaviour of
$\lambda$-polyharmonic functions on regular trees}
\author{\bf Ecaterina Sava-Huss and Wolfgang Woess}
%
\address{\parbox{.8\linewidth}{Institut f\"ur Mathematik \\ 
Universit\"at Innsbruck\\
Technikerstrasse 13, A-6020 Innsbruck, Austria\\ $\,$}} 
\email{Ecaterina.Sava-Huss@uibk.ac.at}

\address{\parbox{.8\linewidth}{Institut f\"ur Diskrete Mathematik,\\ 
Technische Universit\"at Graz,\\
Steyrergasse 30, A-8010 Graz, Austria\\ $\,$}}
\email{woess@tugraz.at}

\date{\today} 
\thanks{Supported by Austrian Science Fund projects FWF P31237 and W1230. 
The second author acknowledges the hospitality of Marc Peign\'e and Kilian Raschel at 
Institut Denis-Poisson, Universit\'e de Tours, France and 
support in the framework of K.R.'s starting grant from the  
European Research Council (ERC) under the Grant Agreement  
No759702.}
\subjclass[2010] {31C20; 
                  05C05, 
                  60G50 
                  }
                  \keywords{Regular tree, simple random walk, $\la$-polyharmonic functions, 
                  Dirichlet and Riquier problems at infinity, Fatou theorem}
\begin{abstract}
This paper studies the boundary behaviour of $\lambda$-polyharmonic functions 
for the simple random walk operator on a regular tree, where $\lambda$ is complex and
$|\lambda|> \rho$, the $\ell^2$-spectral radius of the random walk. In particular,
subject to normalisation by spherical, resp. polyspherical functions, Dirichlet
and Riquier problems at infinity are solved and a non-tangential Fatou theorem is proved.
\end{abstract}

\maketitle
       
\markboth{{\sf Ecaterina Sava-Huss and Wolfgang Woess}}
{{\sf Boundary behaviour of $\lambda$-polyharmonic functions}}
\baselineskip 15pt

\vspace*{-.5cm}
\section{Introduction}\label{sec:intro}

A complex-valued function $f$ on a Euclidean domain $D$ is called \emph{polyharmonic 
of order $n$}, if it satisfies $\Lap^n f \equiv 0$, 
where $\Lap$ is the classical Euclidean Laplacian.
The study of polyharmonic functions originates in  work of the $19^{\text{th}}$ 
century, and is pursued very actively. Basic references are the books
by {\sc Aronszajn, Creese and Lipkin}~\cite{ACL} and by 
{\sc Gazzola, Grunau and Sweers}~\cite{GGS} . 

A classical theorem of 
{\sc Almansi}~\cite{Al} says that if the domain $D$ is star-like with respect
to the origin, then every polyharmonic function of order $n$
has a unique decomposition
$$
f(z) = \sum_{k=0}^{n-1} |z|^{2k}\, h_k(z)\,,
$$
where each $h_k$ is harmonic on $D$, and $|z|$ is the Euclidean length of $z \in D$.
In particular, if the domain is the unit disk
$$
\DD = \{ z= x + \im y \in \C : |z| = \sqrt{x^2 + y^2} < 1 \}, 
$$
then thanks to a Theorem of {\sc Helgason}~\cite{He}, Almansi's decomposition
can be written as an integral representation over the boundary $\partial \DD$ of the disk,
that is, the unit circle, with respect to the \emph{Poisson kernel}
$P(z,\xi) = (1 - |z|^2)/|\xi - z|^2 \;(z \in \DD\,,\; \xi \in \partial \DD)$. Namely, 
\begin{equation}\label{eq:almansi}
f(z) = \sum_{k=0}^{n-1} \int_{\partial \DD} |z|^{2k}\, P(z,\xi)\,d\nu_k(\xi)\,,
\end{equation}
where $\nu_0, \dots, \nu_{n-1}$ are certain distributions, namely
\emph{analytic functionals} on the unit circle. For details on those functionals, see e.g. the nice
exposition by {\sc Eymard}~\cite{Ey}.

A smaller body of work is available on the discrete counterpart, where the
Laplacian is a difference operator arising from a reversible Markov chain transition
matrix on a graph.
Regarding boundary integral representations comparable to \eqref{eq:almansi},
{\sc Cohen et al.}~\cite{CCGS} have provided such a result concerning polyharmonic functions
for the simple random walk operator on a homogeneous tree. This has recently
been generalised by {\sc Picardello and Woess}~\cite{PiWo} to arbitrary nearest neighbour
transition operators on arbitrary trees which do not need to be locally finite: 
\cite{PiWo} provides a boundary integral representation for $\lambda$-polyharmonic
functions for suitable complex $\lambda$. 

Here we come back to the specific situation of simple random walk on the homogeneous
tree $T$ with degree $q+1$, where $q \ge 2$. The necessary preliminaries are 
outlined in \S \ref{sec:integral}.  For the transition operator $P$ of the 
simple random walk on $T$, we study in more detail the boundary behaviour of 
$\lambda$-polyharmonic functions, that is, $f: T \to \C$ such that $(\lambda\cdot I - P)^n f = 0$.
We assume that $\lambda \in \C \setminus [-\rho\,,\,\rho]$, where $\rho$ is the $\ell^2$-spectral 
radius of $P$ and $[-\rho\,,\,\rho]$ is its $\ell^2$-spectrum. Close to the spirit of 
{\sc Kor\'anyi and Picardello}~\cite{KoPi}, we extend their results 
from $\lambda$-harmonic to $\lambda$-polyharmonic functions, and results of the abovementioned
work \cite{CCGS} from ordinary polyharmonic functions, i.e. $\lambda =1$, to general
complex $\lambda$ in the $\ell^2$-resolvent set of $P$.

First, we consider higher order analogues of the Dirichlet problem at infinity: 
in the classical case $\lambda =1$, one takes any continuous
function $\gh$ on the boundary at infinity $\partial T$ of $T$ and provides
a harmonic function on $T$ which provides a continuous extension of $\gh$
to the compactification $\wh T = T \cup \partial T$. It is given by the (analogue of the)
Poisson transform of $\gh$ with respect to the Martin kernel.

However, for $\lambda$-polyharmonic functions of higher order, as well as for
$\lambda$-harmonic functions with $\lambda \ne 1$, this needs an additional normalisation,
in order to control the Poisson-Martin transforms with respect to the $\lambda$-Martin kernel 
(and its higher order versions) at infinity. The normalisation is by
spherical functions and their higher order analogues, the polyspherical functions.
They are introduced in \S \ref{sec:spherical}, where we also study their
asymptotic behaviour at infinity, see Proposition \ref{pro:asy}.

The first two main results are given by the ``twin'' theorems \ref{thm:laDir} and \ref{thm:laFat} 
in \S \ref{sec:DirRiq}. 
The (analogue of the) Poisson integral of $\gh$ with respect to the $n^{\text{th}}$ 
extension of the $\lambda$-Martin kernel (i.e., the  kernel multiplied by the 
-- suitably normalised --
$n^{\text{th}}$ power of the Busemann function) is polyharmonic of order $n+1$,
and normalised (= divided) by the $n^{\text{th}}$ polyspherical function, it converges to
$\gh$ at the boundary. Next,  Theorem \ref{thm:laFat} concerns  
Fatou type non-tangential convergence of polyharmonic extensions of 
complex Borel measures on the boundary.

In general, the polyharmonic extension of a continuous boundary function cannot be 
unique because one may add lower order polyharmonic functions that do not change the limit.
However, uniqueness is proved in the case of $\lambda$-harmonic functions ($n=1$),
see Theorem \ref{thm:launique}. That is, normalising by the associated spherical
function, the solution of the $\lambda$-Dirichlet problem at infinity is unique. 
Note that since $\lambda$ is in general complex, typical tools from Potential
Theory such as the maximum principle cannot be applied here, and are replaced by a new idea,
using spherical averages.

As a corollary of these results, a tree-counterpart of the \emph{Riquier problem} at infinity is
provided. In the case of a bounded Euclidean domain $D$ as above, this consists
in providing continuous boundary functions $\gh_0\,,\dots, \gh_{n-1}$ 
and looking for a polyharmonic function $f$ of order $n$ 
on $D$ such that $\Delta^k f$ is a continuous extension of $\gh_k$ for each $k$.
For finite graphs, the analogous problem has been studied in a note 
by {\sc Hirschler and Woess}~\cite{HW}, where one can find further references
concerning the discrete setting.
In the case of $\lambda$-harmonic functions on $T$, the formulation of the 
analogous problem requires again suitable normalisation, 
see Definition  \ref{def:Riq} and Corollary \ref{cor:Riq}.

\section{Homogeneous trees and boundary integral representations}\label{sec:integral}

Let $T= T_q$ 
be the homogeneous tree where each vertex has $q+1 \ge 3$ neighbours.
We need some features of its structure and first recall the well known boundary 
$\partial T$ of the tree.  For $x, y \in T$, there is a unique \emph{geodesic path}
$\pi(x,y) = [x=x_0\,,\, x_1\,,\dots, x_n=y]$ of minimal length $n$, such that $x_{k-1} \sim x_k$
for $x = 1, \dots,n$, and $d(x,y)=n$ is the graph distance between $x$ and $y$. 
A \emph{geodesic ray}
is a sequence $[x_0\,,\, x_1\,,x_2\,,...]$ of distinct vertices with $x_{n-1} \sim x_n$.
Two rays are equivalent if they share all but finitely many among their vertices. An 
\emph{end} of $T$ is an equivalence class of geodesic rays, and $\partial T$ is
the set of all ends. For any $\xi \in \partial T$ and $x \in T$, there is a unique
geodesic $\pi(x,\xi)$ which starts at $x$ and represents $\xi$.  
Next, we  choose a root vertex $o \in T$. We set $\wh T = T \cup \partial T$. For any pair of points $z, w \in \wh T$, 
their \emph{confluent} $z \wedge w$ is the last common vertex on the finite or 
infinite geodesics $\pi(o,w)$ an $\pi(o,z)$,  unless $z=w$ is an end, in which 
case $z \wedge z = z$. Furthermore, for a vertex $x \ne o$, we define its
\emph{predecessor} $x^-$ as the neighbour of $x$ on the arc $\pi(o,x)$.

We now equip $\wh T$ with a new metric: we set $|x| = d(x,o)$
for $x \in T$, and let
\begin{equation}\label{eq:theta}
\theta(z,w) = \begin{cases} q^{-|z \wedge w|}\,,&\text{if }\; z \ne w\,,\\
                            0\,,&\text{if }\; z = w\,.
              \end{cases}
\end{equation}
This is an ultra-metric which turns $\wh T$ into a compact space with $T$ as an open,
discrete and dense subset. A basis of the topology is given by all \emph{branches}
$\wh T_{x,y}\,$, where $x, y \in T$ with $x \ne y$. Here, 
$$
\wh T_{x,y} = \{ z \in \wh T : y \in \pi(x,z) \}\,.
$$
This is a compact-open set, and its boundary $\partial T_{x,y} = \wh T_{x,y} \cap \partial T$
is called a \emph{boundary arc}. As a matter of fact, a basis of the topology of
$\partial T$ is given by the collection of all $\partial T_x := \partial T_{o,x}\,$,
including $\partial T_o := \partial T$. A \emph{locally constant function} on $\partial T$
is a finite linear combination
$$
\gh = \sum_{j=1}^n c_j\,\uno_{\partial T_{x_j}}
$$
of indicator functions of boundary arcs. It can equivalently be written in terms of
boundary arcs $\partial T_{x,y_k}$ for any fixed vertex $x$. A \emph{distribution} on
$\partial T$ is an element of the dual of the linear space of locally constant functions.
Equivalently, it can be written as a finitely additive measure $\nu$ on the collection of
all boundary arcs. For this it suffices to consider only the boundary arcs with respect
to $o$, so that $\nu$ is characterised as a set function
\begin{equation}\label{eq:distribution}
\nu: \{ \partial T_x : x \in T \} \to \C \quad \text{with}\quad
\nu(\partial T_x) = \sum_{y: y^- = x} \nu(\partial T_y) \quad \text{for all }\;x\,. 
\end{equation}
For $\gh$ as above, we write $\nu(\gh)$ as an integral
$$
\int_{\partial T} \gh \, d\nu = \sum_{j=1}^n c_j \, \nu(\partial T_{x_j})\,.
$$
When $\nu$ is non-negative real, compactness yields immediately that it
extends to a $\sigma$-additive measure on the Borel $\sigma$-algebra
of $\partial T$. In general, $\nu$ does not necessarily extend to
a $\sigma$-additive complex measure; see {\sc Cohen, Colonna 
and Singman}~\cite{CoCoSi}.

We now turn to harmonic functions.
For a function $f: T \to \C$, we define
$$
Pf(x) = \frac{1}{q+1} \sum_{y: y \sim x} f(y)\,,
$$
where $y \sim x$ means that the vertices $x, y \in T$ are neighbours.
$P$ is the transition operator of the simple random walk on $T$. We recall 
the very well known fact that as a self-adjoint operator on the space
$\ell^2(T)$, its spectrum is the interval $[-\rho\,,\,\rho]$, where
$\rho = 2\sqrt{q}/(q+1)$. In this setting, the discrete counterpart of the
Laplacian is $P-I$, where $I$ is the identity operator.

\begin{dfn}\label{def:poly} For $\lambda \in \C$, a  \emph{$\lambda$-polyharmonic 
function of order $n$} is a function $f:T \to \C$ such that
$(\lambda \cdot I - P)^n f = 0$.

For $n=1$, it is called $\lambda$-harmonic,  and when $\lambda =1$, we speak
of a polyharmonic, resp. harmonic function.
\end{dfn}

Following \cite{PiWo}, for a suitable boundary integral representation, 
the ``eigenvalue'' $\lambda$ should belong to the resolvent set
$\res(P) = \C \setminus [-\rho\,,\,\rho]$ of $P$ on $\ell^2(T)$. In this case, let
$G(x,y|\lambda) = (\lambda \cdot I - P)^{-1} \delta_y(x)$ be
the Green function, that is, the $(x,y)$-matrix element of the resolvent,
where $x,y \in T$. By \cite[Thm. 4.2]{PiWo}, or by direct computation,
$G(x,x|\la) \ne 0$, and we can define $F(x,y|\la) = G(x,y|\la)/G(x,x|\la)$.
These functions depend only on the graph distance $d(x,y)$ between $x$ and $y$.

For $|\la| \ge \rho$, one has a combinatorial-probabilistic interpretation:
\begin{equation}\label{eq:copr}
F(x,y|\la) = \sum_{n=1}^{\infty} f^{(n)}(x,y)/\lambda^n\,,
\end{equation}
where  $f^{(n)}(x,y)$ is the probability that the simple random walk starting
at $x$ hits $y$ at the $n^{\text{th}}$ step for the first time. 
Simple and well-known computations yield
\begin{equation}\label{eq:F}
\begin{gathered}
F(x,y|\lambda) = F(\lambda)^{d(x,y)}, \quad \text{where} \\
F(\lambda) = \frac{q+1}{2q} \bigl(\lambda - s(\lambda)\bigr)
\quad\text{with}\quad s(\lambda) = \lambda \, \sqrt{1- \rho^2/\lambda^2}\,,
\end{gathered}
\end{equation}
see e.g. \cite[Lemma 1.24]{Wbook} (with $z=1/\lambda$). The complex square root is
$\sqrt{re^{i\phi}} = \sqrt{r}e^{i\phi/2}$ for $\phi \in (-\pi\,,\,\pi)$.  

The \emph{$\lambda$-Martin kernel}
on $T \times \partial T$ is 
$$
K(x,\xi|\lambda) = 
\frac{F(x,x\wedge \xi|\lambda)}{F(o,x \wedge \xi|\lambda)} = F(\lambda)^{\hor(x,\xi)}\,,
\quad x \in T\,,\;\xi \in \partial T\,, 
$$
where 
$$
\hor(x,\xi) = d(x,x\wedge \xi) -d(o, x \wedge \xi)
$$
is the \emph{Busemann function} or \emph{horocycle index} of $x$ with respect to the end $\xi$.
Note that for fixed $x$, the function $\xi \mapsto K(x,\xi|\lambda)$ is locally constant.

Now a basic result in the seminal paper of {\sc Cartier}~\cite{Ca}, valid for real 
$\lambda \ge \rho$, and its extension to complex $\lambda \in \res(P)$ \cite{PiWo} 
says the following for simple random walk on $T$. 

For $\lambda \in \C \setminus [-\rho\,,\,\rho]$, 
every $\lambda$-harmonic function
$h$ on $T$ has a unique integral representation
\begin{equation}\label{eq:Cartier}
h(x) = \int_{\partial T}  K(x,\xi|\lambda)\, d\nu(\xi)\,,
\end{equation}
where $\nu$ is a distribution on $\partial T$ as in \eqref{eq:distribution}.
If $\lambda > \rho$ and $h > 0$ then $\nu$ is a positive Borel measure.
Indeed, this holds for arbitrary nearest neighbour random walks on arbitrary
countable trees, and \cite{PiWo} has a method to extend this to a
boundary integral representation of $\lambda$-polyharmonic functions. Specialised
to simple random walk on $T = T_q\,$, this yields the following extension of
a result of \cite{CCGS}, where the basic case $\lambda =1$ is considered.

\begin{theorem}\cite{PiWo} \label{thm:PiWo}
For $\lambda \in \C \setminus [-\rho\,,\,\rho]$, 
every $\lambda$-polyharmonic harmonic function $f$ of order $n$
on $T$ has a unique integral representation
$$
f(x) = \sum_{k=0}^{n-1}\int_{\partial T}  K(x,\xi|\lambda)\,
\hor_k(x,\xi|\lambda)\, d\nu_k(\xi) \quad \text{with}\quad 
\hor_k(x,\xi|\lambda) = \frac{\hor(x,\xi)^k}{k!\,s(\lambda)^k}\,, 
$$
where $\nu_0\,,\dots, \nu_{n-1}$ are distributions on $\partial T$.
\end{theorem}

The normalisation by $k!\,s(\lambda)^k$, where $s(\lambda)$ is as in
\eqref{eq:F}, is not present in \cite[Cor. 5.4]{PiWo}. We shall see below in
Lemma \ref{lem:new} 
why it is useful.

\section{Polyspherical functions}\label{sec:spherical}

\begin{dfn}\label{def:sph}
For any $\lambda \in \C$, the \emph{spherical function}
$\Phi(x|\lambda)$ is the unique function on $T$ with $\Phi(o|\lambda)=1$
which is $\lambda$-harmonic and \emph{radial,} i.e., it depends only on $|x|= d(o,x)$.
\end{dfn}

Namely, if we set $\varphi_k(\lambda) = \Phi(x|\lambda)$ for $|x|=k$, then 
we have the recursion 
$$
\varphi_0(\lambda)=1\,,\quad \varphi_1(\lambda) = \lambda\,,\AND
\lambda\,\varphi_k(\lambda) = \frac{1}{q+1}\,\varphi_{k-1}(\lambda) 
+ \frac{q}{q+1}\,\varphi_{k+1}(\lambda) \quad \text{for }\; k \ge 1\,. 
$$
We shall consider the case when $\lambda$ is in the $\ell^2$-resolvent set of $P$, 
that is, $\lambda \in \C \setminus [\rho\,,\,\rho]$. 
Let $F(\la)$ be as in \eqref{eq:F}, and let 
\begin{equation}\label{eq:oF}
 \Fc(\lambda) = \frac{q+1}{2q} \bigl(\lambda + s(\lambda)\bigr)
\end{equation}
be the second solution, besides $F(\lambda)$, of the equation 
\begin{equation}\label{eq:qeq}
\lambda\,F(\lambda) = \frac{1}{q+1}
+ \frac{q}{q+1}F(\lambda)^2\,.
\end{equation}
Then one can solve the above recursion, and 
\begin{equation}\label{eq:solverec}
\begin{gathered} 
\Phi(x|\lambda) 
= a(\lambda)\,F(\lambda)^{|x|} + \ac(\lambda)\, \Fc(\lambda)^{|x|}\,,\quad \text{where}\\
a(\lambda) = \frac{s(\lambda)-\frac{q-1}{q+1}\lambda}{2s(\lambda)}
   \AND \ac(\lambda) = \frac{s(\lambda)+\frac{q-1}{q+1}\lambda}{2s(\lambda)}\,.
\end{gathered}
\end{equation}
We collect a few elementary properties.
\begin{lem}\label{lem:properties} We have for $\lambda \in \C \setminus [\rho\,,\,\rho]$
$$
0 < |F(\lambda)| < 1\big/\sqrt{q} < |\Fc(\lambda)|\,, \quad F(1) = 1/q\,, \AND \Fc(1)=1.  
$$ 
Furthermore, 
$$
\Phi(x|1) = 1 \AND \Phi(x|\lambda) \ne 0 \quad \text{for all }\; x \in T.
$$
\end{lem}

\begin{proof}
 First of all, by \eqref{eq:qeq}, $F(\lambda) \ne 0$ and $F(\lambda)\Fc(\lambda) = 1/q$. 
Next, by \eqref{eq:copr}, $|F(\lambda)| \le F(|\lambda|) < F(\rho) = 
1\big/\sqrt{q}\,$ for $|\lambda| > \rho$. Also when $|\lambda| = \rho$ and $\la \ne \pm \rho$,
we have $|F(\lambda)| < F(\rho) = 1\big/\sqrt{q}\,$. At last, for $\lambda$ in the real interval
$(-\rho\,,\rho)$, the limits of $F(\cdot)$ are 
$$
\frac{q+1}{2q} \bigl(\lambda \pm \im \sqrt{\rho^2 - \lambda^2}\bigr),
$$
according to whether $\lambda$ is approached within the upper or lower half plane.
Thus, in the upper open semidisk $\{ z \in \C : |z| < \rho\,,\; \Re z > 0 \}$, 
as well as in the corresponding lower open semidisk, $F(\lambda)$ is analytic, and 
its absolute values at the boundary are $\le 1\big/\sqrt{q}\,$. 
By the Maximum Modulus Principle, $|F(\lambda)| <1/\sqrt{q}\,$ within each of those
two semidisks.  We see that the last inequality holds in all of $\C \setminus [\rho\,,\,\rho]$.

Consequently, $|q \Fc(\lambda)| = 1/|F(\lambda)| > \sqrt{q}$.
The values for $\lambda=1$ are obvious. 

Finally, we claim that for the coefficient functions in 
\eqref{eq:solverec} one has $|a(\lambda)| < |\ac(\lambda)|$. For $|\lambda| > \rho$, as well as 
for $|\lambda| = \rho$ and $\lambda \ne \pm \rho$,  one can see this
from the fact that $1- \rho^2/\lambda^2$ belongs to the complex half-plane with
positive real part. For $\lambda$ in one of the above two semidisks, one can proceed as above:
one checks that $\ac(\lambda) \ne 0$. Then the function $a(\lambda)/\ac(\lambda)$ is analytic
in each semidisk, with boundary values whose absolute values are $\le 1$, and the 
desired inequality follows. 
Therefore 
$$
\bigl|a(\lambda)\,F(\lambda)^{|x|}\bigr| < \bigl|\ac(\lambda)\, \Fc(\lambda)^{|x|}\bigr|
$$
for every $x \in T$, and $\Phi(x|\lambda) \ne 0$.
\end{proof}

We can describe the spherical functions via their integral representation \eqref{eq:Cartier}.
Let $\mm$ stand for the \emph{uniform distribution} on $\partial T$. This is the Borel probability 
measure which for each $k \in \N_0$ assigns equal mass to all boundary arcs $\partial T_x\,$,
where $x \in T$ with $|x|=k$. That is, 
$$
\mm(\partial T_x) 
= \begin {cases} 1\,,&\text{if }\; x = o\,,\\
                 1\big/\bigl((q+1)q^{|x|-1}\bigr)\,, &\text{if }\; x \ne o\,.
\end{cases}
$$
We shall often write $d\mm(\xi) = d\xi\,$. Then
\begin{equation}\label{eq:sph}
\Phi(x|\lambda) = \int_{\partial T} K(x,\xi|\lambda)\, d\xi\,. 
\end{equation}
Indeed, the right hand side satisfies all requirements of Definition \ref{def:sph},
which determine the spherical function. A comparison with Theorem \ref{thm:PiWo}
leads us to the following.

\begin{dfn}\label{def:polysph}
For $n \ge 0$,  the \emph{$n^{\text{th}}$ polyspherical function} is
$$
\Phi_n(x|\lambda) = \int_{\partial T}  K(x,\xi|\lambda)\,\hor_n(x,\xi|\lambda)
\,  d\xi\,.
$$
\end{dfn}
It is $\lambda$-polyharmonic of order $n+1$, and it is radial. With respect to those two properties,
it is uniquely determined by its values for $|x|=0, 1, \dots, n$. For $n \ge 1$, 
its value at $x=o$ is $0$.
For $n=0$ it is of course the spherical function  \eqref{eq:sph}.

In particular, $(\lambda\cdot I - P)^n\, \Phi_n(\cdot|\lambda)$ is $\lambda$-harmonic and
radial, so that it must be a multiple of $\Phi(\cdot|\lambda)$. In order to determine
the factor, we need to recall part of how Theorem \ref{thm:PiWo} was obtained in \cite{PiWo}.
Let $K^{(n)}(x,\xi|\lambda)$ be the $n^{\textrm{th}}$ derivative of $K(x,\xi|\lambda)$ with
respect to $\lambda$. Then 
\begin{equation}\label{eq:derive}
\frac{(-1)^n}{n!}(\lambda\cdot I - P)^n K^{(n)}(\cdot,\xi|\lambda) = K(\cdot,\xi|\lambda)\,.
\end{equation}
In \cite[equation (5.2)]{PiWo}, it is shown that 
\begin{equation}\label{eq:rec}
K^{(n)}(x,\xi|\la) = K(x,\xi|\la)\, \sum_{k=1}^n \hor(x,\xi)^k \, \ff_{k,n}(\la)\,,
\end{equation}
where the functions $\ff_{k,n}(\lambda)$ are given recursively; in particular, with 
$s(\lambda)$ as in \eqref{eq:F},
$$
\ff_{n,n}(\lambda) = (-1)^n s(\la)^{-n} \,.
$$
Combining \eqref{eq:derive} and  \eqref{eq:rec}, we get
\begin{lem}\label{lem:new} \hspace*{1.3cm}
$(\lambda\cdot I - P)^n 
\bigl[K(\cdot,\xi|\lambda)\, \hor_n(x,\xi|\lambda) \bigr] = K(\cdot,\xi|\lambda)\,.$ 
\end{lem}

Integrating with respect to $d\xi$, we also obtain the following.
\begin{equation}\label{eq:sphn} 
(\lambda\cdot I - P)^n \Phi_n(\cdot|\lambda) =  \Phi(\cdot|\lambda)\,.
\end{equation}

We shall need the asymptotic behaviour of $\Phi_n(x|\la)$ as $|x| \to \infty$. 

\begin{pro}\label{pro:asy}
Let $\lambda \in \C$ with $|\lambda| > \rho$. Then, as $|x| \to \infty$,
$$
\Phi_n(x|\lambda) 
\sim \ac(\lambda)\, \frac{(-1)^n\, |x|^n}{n! \, s(\lambda)^n}\, \Fc(\lambda)^{|x|}\,,
$$ 
with $\ac(\lambda)$ given by \eqref{eq:solverec}. 
In particular, in the standard case $\lambda =1$, we have $\Fc(1) = \ac(1) = 1$.

Therefore there is $R = R_{n,\lambda} > 0$ such that 
$$
\Phi_n(x|\lambda) \ne 0  \AND |\Phi_n(x|\lambda)| \le 
2 |\ac(\lambda)|\, \frac{|x|^n}{n! \, |s(\lambda)|^n}\, |\Fc(\lambda)|^{|x|}
\quad \text{for }\; |x| \ge R\,.
$$ 
Furthermore,
\begin{equation}\label{eq:asy}
\lim_{|x| \to \infty} \frac{\Phi_k(x|\lambda)}{\Phi_n(x|\lambda)} = 0 
\quad \text{for }\; k < n\,.
\end{equation}
\end{pro}

\begin{proof} By Lemma \ref{lem:properties},
$$
| q \, \Fc(\lambda)^2| > 1 \quad \text{for} \quad |\lambda| > \rho.
$$
Now let $x \in T \setminus \{o\}$. For $\ell \in \{ 0, 1, \dots, |x|\}$,
let $A_{\ell} = \{ \xi \in \partial T : |x \wedge \xi| = \ell \}$, and set 
$\mm_{\ell} = m(A_{\ell})$. Then 
$$
K(x,\xi|\la) = F(\la)^{|x|-2\ell}  
 \; \text{ for }
\xi \in A_{\ell}\,,\AND \mm_{\ell} 
= \begin{cases} \dfrac{q}{q+1}&\text{for}\;\ell =0\,,\\[10pt]
                \dfrac{q-1}{(q+1)q^{\ell}}&\text{for}\;\ell =1,\dots, |x|-1\,,\\[10pt]
                \dfrac{1}{(q+1)q^{|x|-1}}&\text{for}\;\ell =|x|\,.
  \end{cases}
$$
We use $F(\lambda) = \bigl(q \Fc(\lambda)\bigr)^{-1}$ and set $k=|x|-\ell$. 
Then the integral formula of Definition \ref{def:polysph} translates into
$$
\begin{aligned}
&n! \, s(\lambda)^n\,\Phi_n(x|\lambda) = \sum_{\ell=0}^{|x|}  F(\la)^{|x|-2\ell} \,\bigl(|x|-2\ell\bigr)^n\,\mm_{\ell}\\
&=  \frac{q}{q+1} \, \bigl(-|x|\bigr)^n\, \Fc(\lambda)^{|x|} \left(1 + 
\frac{q-1}{q}\sum_{k=1}^{|x|-1} \bigl( q\,\Fc(\lambda)^2\bigr)^{-k} 
\Bigl(\frac{|x|-2k}{|x|}\Bigr)^{\!n}
+ (-1)^n \bigl( q\,\Fc(\lambda)^2\bigr)^{-|x|}\right).
\end{aligned}
$$
The last term within the big parentheses tends to $0$ as $|x| \to \infty\,$.
Decompose the sum into the two pieces where in the first one, summation is over
$1 \le k \le \sqrt{|x|}$ and in the second one, summation is over $k > \sqrt{|x|}$.
Then the second part is a remainder of a convergent series, so that it also tends to $0$
as $|x| \to \infty$. Now, in the range $1 \le k \le \sqrt{|x|}$, the quotients 
$(|x|-2k)/|x|$ tend to $1$ uniformly as $|x| \to \infty$. Therefore the first part of the sum
converges to
$$
\frac{q-1}{q} \sum_{k=1}^{\infty} \bigl( q\,\Fc(\lambda)\bigr)^{-k}  
= \frac{q-1}{q} \frac{1}{q\,\Fc(\lambda)^2-1}\,,
$$
as $|x| \to \infty$. This yields the proposed asymptotic formula, with some elementary
computations for getting the factor $\ac(\lambda)$. 
\end{proof}

\section{Dirichlet, Riquier and Fatou type convergence}\label{sec:DirRiq}

In the classical case of harmonic functions, that is, when $\lambda =1$,
the Dirichlet problem  
asks whether for any real or complex valued function $\gh \in \mathcal{C}(\partial T)$,
there is a continuous extension to $\wh T$ which is harmonic in $T$.
That is, we look for a function $h = h_{\gh}$ on $T$ such that 
$$
(I-P)h = 0 \AND \lim_{x \to \xi} h(x) = \gh(\xi) \quad \text{for every }\; \xi \in \partial T.
$$
If a solution exists then it is necessarily unique by the minimum (maximum) principle.
For our simple random walk on $T$, it is folklore that the Dirichlet problem is solvable,
and that the solution is given as the \emph{Poisson integral} of $\gh\,$:
$$
h(x) = \int_{\partial T} K(x,\xi|1)\, \gh(\xi)\, d\xi\,.
$$
We are now interested in the general case when $\la \in \C \setminus [-\rho\,,\,\rho]$, which will
remain fixed throughout this section.
First of all, the above question is not well-posed. Indeed, if for example $\lambda > 1$
is real, then  the ``Poisson integral'' of the constant function
$\uno$ on $\partial T$ is $\Phi(x|\lambda)$. By Proposition \ref{pro:asy}, it tends to $\infty$
as $|x| \to \infty$, since $\Fc(\lambda) > 1$. 
Thus, we need to normalise, compare with \cite{KoPi}. 
The same is necessary for the polyharmonic versions of higher order.

\begin{theorem}\label{thm:laDir} 
Let $\lambda \in \C$ with $|\lambda| > \rho$. For $\gh \in \mathcal{C}(\partial T)$ and $n \ge 0$, set 
$$
f(x) = \int_{\partial T}  K(x,\xi|\lambda)\,\hor_n(x,\xi|\lambda) 
\, \gh(\xi)\, d\xi
$$
Then $f$ is $\lambda$-polyharmonic of order $n+1$ and 
\begin{equation}\label{eq:asymp}
\lim_{x \to \xi} \frac{f(x)}{\Phi_n(x|\la)} = \gh(\xi)  \quad \text{for every }\; 
\xi \in \partial T.
\end{equation}
\end{theorem}

Before the proof of this result, we introduce the normalized kernel
\begin{equation}\label{eq:kernel}
\KK_n(x,\xi|\la) = \frac{K(x,\xi|\lambda)\,\hor_n(x,\xi|\lambda)}{\Phi_n(x|\lambda)}\,,
\quad n \ge 0.
\end{equation}
We only need it for large $|x|$, and then $\Phi_n(x|\la) \ne 0$ 
by Proposition \ref{pro:asy}, so that the division in \eqref{eq:asymp} and 
the definition of $\KK_n$ are legitimate. If we fix such an $x \in T$ with $|x| \ge R$, 
the function 
$\xi \mapsto \KK_n(x,\xi|\lambda)$ is locally constant, since it depends only on $x \wedge \xi$ 
which ranges within the finite geodesic $\pi(o,x)$. Therefore it is continuous.
\begin{lem}\label{lem:unif} Let $y \in T$. Then
$$
\lim_{|x| \to \infty\,,\, x \in \partial T_y} \KK_n(x,\xi|\lambda) = 0
$$
uniformly for $\xi \in \partial T \setminus \partial T_y\,$.
\end{lem}

\begin{proof}
If $x \in T_y$  and $\xi \in \partial T \setminus \partial T_y$ then 
$x \wedge \xi = y \wedge \xi \in \pi(o,y)$.
We have 
$$
\hor(x,\xi) = |x| - 2|y \wedge \xi| \ge |x| - 2|y|\,.
$$
Therefore, using Lemma \ref{lem:properties} and Proposition \ref{pro:asy},
$$
\KK_n(x,\xi|\lambda)
\sim 
\frac{|F(\lambda)|^{-2|y\wedge \xi|}}{|\ac(\lambda)|} \,
\left|\frac{F(\lambda)}{\Fc(\lambda)}\right|^{|x|}\, \left(\frac{|x|-2|y\wedge \xi|}{|x|}\right)^{n}\,,
$$
which tends to $0$ as proposed.
\end{proof}

\begin{proof}[Proof of Theorem \ref{thm:laDir}.]
For $x \in T$ with $|x| \ge R$, 
$$
d\mu_x(\xi) = \KK_n(x,\xi|\lambda)\,d\xi
$$
defines a complex Borel measure on $\partial T$. (It also depends on $\lambda$ and $n$, which we
omit in the present notation.) We have
$
\mu_x(\partial T) = 1\,.
$ 
We write $|\mu|_x$ for its total variation measure. Its density with
respect to $d\xi$ is
$\bigl|K(x,\xi|\lambda)\,\hor_n(x,\xi|\lambda)\bigr|\big/\bigl|\Phi_n(x|\lambda)\bigr|$.
Let us write 
$$
|\Phi|_n(x|\lambda) = \int_{\partial T} \bigl|K(x,\xi|\lambda)\,\hor_n(x,\xi|\lambda)\bigr|\,d\xi\,. 
$$
A computation completely analogous to the one in the proof of Proposition \ref{pro:asy} shows
that 
$$
|\Phi|_n(x|\lambda) \sim C(\lambda)\, \frac{|x|^n}{n!\,|s(\lambda)|^n}\,|\Fc(\lambda)|^{|x|}
\,,\quad \text{where} \quad C(\lambda) 
= \frac{1}{q+1}\, \frac{q^2 |\Fc(\lambda)|^2 - 1}{q |\Fc(\lambda)|^2 - 1}.
$$
Therefore
\begin{equation}\label{eq:as}
|\mm|_x(\partial T) = \frac{|\Phi|_n(x|\lambda)}{|\Phi_n(x|\la)|} \to 
\frac{C(\lambda)}{|\ac(\lambda)|}\,, \quad \text{as }\; |x| \to \infty\,.
\end{equation}
We can now prove \eqref{eq:asymp} along classical lines. Let $\xi_0 \in \partial T$ and $\ep > 0$. 
Then, given $\gh \in C(\partial T)$, there is a neighbourhood of $\xi_0$ on which 
$|\gh(\xi)-\gh(x_0)|<\ep$.
We may assume that this neighbourhood is of the form $\partial T_y\,$, where $y \in \pi(o,\xi_0)$.
If $x \to \xi_0$ then $x \in T_y$ when $|x|$ is sufficiently large.  Then
$$
\left|\frac{f(x)}{\Phi_n(x|\la)} - \gh(\xi_0)\right| =
\left|\int_{\partial T} \Bigl(\gh(\xi)-\gh(\xi_0)\Bigr)\,
 \, d\mu_x(\xi)\right| \le  2\|\gh\|_{\infty}\, |\mu|_x(\partial T\setminus \partial T_y)
+ \ep \, |\mm|_x(\partial T_y)\,.
$$ 
Now Lemma \eqref{lem:unif} implies that for $x \to \xi$ we have 
$
|\mu|_x(\partial T\setminus \partial T_y) \to 0\,,
$
while $|\mu|_x(\partial T_y)$ remains bounded by Lemma \eqref{eq:as}.
\end{proof}

Next, we consider a Fatou-type theorem for polyharmonic functions. That is, in the integral
of Theorem \ref{thm:laDir}
we replace $\gh(\xi)\, d\xi$ by a complex Borel measure $\nu$ on $\partial T$. We need to 
consider a restricted type of convergence to the boundary. 

\begin{dfn}\label{def:cone} 
Let $\xi \in \partial T$ and $a \ge 0$. The \emph{cone} at $\xi$ of width $a$ is
$$
\Gamma_a(\xi) = \bigl\{ x \in T : d\bigl(x,\pi(o,\xi)\bigr) \le a \bigr\}. 
$$
\end{dfn}
The motivation for this definition is well-known: in the open unit disk, consider a 
cone $C_{\alpha}(z)$  whose vertex is a point $z$ on the unit circle, whose axes
connects the origin with $z$, and whose opening angle is $\alpha < \pi$.
Then, passing to the hyperbolic metric on the disk, all elements of the cone 
are at bounded distance (depending on $\alpha$) from the axes.  The standard graph metric of
$T$ should be seen as an analogue of the hyperbolic metric on the disk, while a tree-analogue
of the Euclidean metric is the ultrametric $\theta$ of \eqref{eq:theta}. Compare with 
{\sc Boiko and Woess}~\cite{BW} for a ``dictionary'' concerning the many of the other analogies 
between the potential theory on the unit disk and $T$. Thus, $a$ is a substitute for the 
angle $\alpha$, and of course, if $|x| \to \infty$ within $\Gamma_a(\xi)$ then
$x \to \xi$ in the topology of $\wh T$. We shall use the following tools. 

\begin{lem}\cite{KoPi}\label{lem:max}
For $\gh \in L^1(\partial T,\mm)$, let 
$$
\mathcal{M}\gh(\xi) = \sup_{x \in \pi(o,\xi)} 
\frac{1}{\mm(\partial T_x)} \int_{\partial T_x} |\gh| \, d\mm
$$
be the associated \emph{Hardy-Littlewood maximal function} on $\partial T$. Then the 
operator $\gh \mapsto \mathcal{M}\gh$ is weak type (1,1), that is, there is $C >0$
such that for every $t > 0$, 
$$
\mm\bigl[\, |\mathcal{M}\gh| \ge t\bigr] \le C\, \|\gh\|_1\big/t \quad \text{for all }\; 
\gh \in L^1(\partial T,\mm). 
$$
\end{lem}

With $R$ as in Proposition \ref{pro:asy}, we now define for $a \ge 0$ and 
$\gh \in L^1(\partial T,\mm)$,
\begin{equation}\label{eq:maxa}
\mathfrak{M}_a \gh (\xi) = \sup 
\left\{ \left|\int_{\partial T} \mathcal{K}_n(x,\cdot|\lambda)\, \gh \, d\mm\right| : 
x \in \Gamma_a(\xi)\,,\; |x| \ge R \right\}.
\end{equation}

\begin{pro}\label{pro:max} For every $a \ge 0$ there is a constant $C_a$ such that
$$
\mathfrak{M}_a \gh \le C_a \,\mathcal{M}\gh
\quad \text{for every }\; 
\gh \in L^1(\partial T, \mm).
$$ 
\end{pro}

\begin{proof} Let $\pi(o,\xi)= [o=x_0\,,x_1\,,x_2\,,\dots].$
First, let $a = 0$. Fix $x = x_r$ with $r \ge R$.
Then,  with $A_{\ell} = \{ \eta \in \partial T : |x \wedge \eta| = \ell \}$
as above, we use the properties listed in Lemma \ref{lem:properties} and
compute
$$
\begin{aligned}
\left|\int_{\partial T} \mathcal{K}_n(x,\cdot|\lambda)\, \gh\, d\mm\right| 
&\le 2 \, \int\limits_{\partial T} 
\frac{\bigl|K(x,\cdot|\la)\, \hor(x,\cdot)^n\bigr|}{|\ac(\lambda)| \,\,|x|^n\,|\Fc(\lambda)|^{|x|}} 
\,|\gh| \,  d\mm\\
&= \frac{2}{|\ac(\lambda)|} \,\left|\frac{F(\lambda)}{\Fc(\lambda)}\right|^{|x|}
\sum_{\ell=0}^{|x|} \int\limits_{A_{\ell}} |F(\lambda)|^{-2\ell}\,
\left(\frac{|x|-2\ell}{|x|}\right)^n |\gh| \, d\mm\\
&\le \frac{2}{|\ac(\lambda)|}\,\left|\frac{F(\lambda)}{\Fc(\lambda)}\right|^{|x|}
\sum_{\ell=0}^{|x|} \,\int\limits_{\partial T_{x_{\ell}}}|F(\lambda)|^{-2\ell}\,|\gh| \, d\mm\\
&\le \frac{2(q+1)}{q\,|\ac(\lambda)|}\,
\sum_{\ell=0}^{|x|} q^{-\ell}|F(\lambda)|^{-2\ell}\,\mathcal{M}g(\xi)\\
&= \frac{2(q+1)}{q\,|\ac(\lambda)|}\,\sum_{\ell=0}^{|x|} 
\left|\frac{F(\lambda)}{\Fc(\lambda)}\right|^{|x|-\ell}\,\mathcal{M}g(\xi)\\ 
&\le C_0\,\mathcal{M}g(\xi)\,,\qquad \text{where} \quad 
C_0 = \frac{2(q+1)}{q\,|\ac(\lambda)|}\,\frac{1}{1-|F(\lambda)/\Fc(\lambda)|}\,.
\end{aligned}
$$
For general $a \in \N$, let  $y \in T$ with $|y| \ge R$ and $d\bigl(y, \pi(o,\xi)\bigr) \le a$.
Then $d(x,y) \le 2a$, where $x$ is the element on $\pi(o,\xi)$ with $|x|=|y|$.
Recall that $|F(\lambda)| <1$. 
Since $|\hor(x,\eta) - \hor(y,\eta)| \le d(x,y)$, we have 
$$
|K(y,\eta|\lambda)| = |F(\lambda)|^{\hor(y,\eta)} \le |F(\lambda)|^{-2a} K(x,\eta|\lambda)
\AND |\hor(y,\eta)|^n \le (1+2a)^n |\hor(x,\eta)|^n \,,
$$
for every $\eta \in \partial T$. Therefore 
$$
|\mathcal{K}_n(y,\cdot|\lambda)| \le (1+2a)^n\,|F(\lambda)|^{-2a}\,|\mathcal{K}_n(x,\cdot|\lambda)|\,.
$$
Setting $C_a = (1+2a)^n\,|F(\lambda)|^{-2a}\, C_0\,$, the proposition follows. 
\end{proof}

After Lemma \ref{lem:max} and Proposition \ref{pro:max}, also
the proof of the following theorem now follows the strategy of \cite{KoPi}. 
For the sake of providing a complete picture in the situation of trees,
we also include some of the ``standard'' details in its proof.

\begin{theorem}\label{thm:laFat} 
Let $\lambda \in \C$ with $|\lambda| > \rho$, and let  $\nu$ be a complex Borel measure on
$\partial T$. For $n \ge 0$, set 
$$
f(x) = \int_{\partial T} K(x,\xi|\lambda)\,\hor_n(x,\xi|\lambda) 
\, d\nu(\xi).
$$
Then $f$ is $\lambda$-polyharmonic of order $n+1$ and  
\begin{equation}\label{eq:Fat}
\lim_{x \to \xi\,,\, x \in \Gamma_a(\xi)} \frac{f(x)}{\Phi_n(x|\la)} = \gh(\xi)  \quad 
\text{for every $\;a \ge 0\;$ and $\mm$-almost every }\; 
\xi \in \partial T,
\end{equation}
where $\gh$ is the Radon-Nikodym derivative of the absolutely continuous part of $\nu$
with respect to the uniform distribution $\mm$ on $\partial T$.
\end{theorem}

\begin{proof}
We first give an outline of the standard fact that the limit in \eqref{eq:Fat} is $0$ when
$\nu$ is singular with respect to equidistribution. The latter means that there is a 
Borel set $E \subset \partial T$ with uniform measure $0$ such that $\partial T \setminus E$
is a $\nu$-null-set. For every $\ep > 0$ there are  disjoint boundary arcs 
$\partial T_{y_1}\,,\dots, \partial T_{y_k}$ depending on $\ep$, whose union $E_{\ep}$ 
contains $E$  and has uniform measure $< \ep$. Let $|\nu|$ be the total variation measure
of $\nu$.
If $x \to \xi_0 \in \partial T \setminus E_{\ep}\,$, then by Lemma \ref{lem:unif},
$$
\left|\frac{f(x)}{\Phi_n(x|\lambda)}\right| \le \int_{E_{\ep}}   
|\KK_n(x,\xi|\lambda)|
\,d|\nu|(\xi) \to 0\,.
$$
Since this holds for every $\ep > 0$, we get that $f(x)/\Phi_n(x|\lambda) \to 0$ 
almost everywhere on $\partial T$.

\smallskip

Now we may assume without loss of generality that we have 
$\gh = d\nu /d\mm \in L^1(\partial  T,\mm)$. Then there is a sequence $(\gh_k)_{k \in \N}$
of continuous functions on $\partial T$ such that
$$
\sum_k \| \gh - \gh_k \|_1 < \infty\,.
$$ 
Set 
$$
f_k(x) = \int_{\partial T} K(x,\xi|\lambda)\,\hor_n(x,\xi|\lambda)\,\gh_k(\xi)\,d\xi\,.
$$
By Lemma \ref{lem:max} and Proposition \ref{pro:max},
$$
\sum_k \mm \bigl[ \mathfrak{M}_a(\gh-\gh_k) \ge \ep \bigr] 
\le C_a\, C \sum_k \| \gh - \gh_k \|_1/\ep 
< \infty 
$$
for every $\ep > 0$. By the Borel-Cantelli Lemma, this yields that 
$$
\mm(A) = 1\,,\quad\text{where}\quad 
A = \Bigl\{ \xi \in \partial T : \lim_{k \to \infty} \mathfrak{M}_a(\gh-\gh_k)(\xi)  = 0 \Bigr\}.
$$
For each $k$, the function on $\wh T$
with values $\gh_k(\xi)$ for $\xi \in \partial T$ and $f_k(x)/\Phi_n(x|\lambda)$
for $x \in T$ is continuous on $\wh T$ by Theorem \ref{thm:laDir}. This readily implies that for
$\xi \in A$, we have convergence as proposed in \eqref{eq:Fat}.
\end{proof}

We now come back to continuous boundary functions and Theorem \ref{thm:laDir}.
For $n \ge 1$, we cannot expect uniqueness of $f$ as a polyharmonic function of order $n+1$
which has the asymptotic behaviour of \eqref{eq:asymp}. Indeed, \eqref{eq:asy} shows that we can add
polyharmonic functions of lower order such that the limit in Theorem \ref{thm:laDir}
remains the same. However, for the case $n=0$, i.e., for $\lambda$-harmonic
functions, we can investigate uniqueness: this case corresponds to the classical
Dirichlet problem at infinity. Indeed, for real $\lambda > \rho$ one can use
the typical argument, namely the maximum principle, to prove uniqueness.
However, for complex $\lambda$, this is not available, and we have to introduce
another method.

\begin{theorem}\label{thm:launique} 
Let $\lambda \in \C \setminus [-\rho\,,\,\rho]$. 
For $\gh \in \mathcal{C}(\partial T)$, the function
$$
h_{\gh}(x) = \int_{\partial T} \gh(\xi)\, K(x,\xi|\lambda)\, d\xi
$$
is the \emph{unique} solution of the $\lambda$-Dirichlet problem with boundary
function $\gh$, i.e., the unique $\lambda$-harmonic function such that
$$
\lim_{x \to \xi} \frac{h_{\gh}(x)}{\Phi(x|\la)} = \gh(\xi)  \quad \text{for every }\; 
\xi \in \partial T.
$$
\end{theorem}

\begin{proof} 
Continuity holds by Theorem \ref{thm:laDir}.
By linearity, we need to prove uniqueness only in the case when $\gh \equiv 0$.
Thus, we assume that $\lambda \cdot h = Ph$ and that 
$\lim_{|y| \to \infty} h(y)/\Phi(y|\lambda) = 0$, and we have to show that $h \equiv 0$.

We extend the notion of the spherical functions as follows:  
$$
\Phi(x,y|\lambda) = \varphi_{d(y,x)}(\lambda)\,, 
$$
where the functions $\varphi_k$ are given by \eqref{eq:solverec}. For fixed $x \in T$, 
this is the unique $\lambda$-harmonic function of $y$ with value $1$ at $x$ which is 
radial with respect to the point $x$.  
Now let us define the spherical average of $h$  around $x$, that is,
the function defined by
$$
\bar h(x) = h(x)\AND \bar h(y) = \frac{1}{(q+1)q^{d(y,x)-1}} \sum_{v \,:\, d(v,x) = d(y,x)} h(v)\,,
\; \text{ if }\;y \ne x.
$$
A short computation shows that $\bar h$ is $\lambda$-harmonic, whence
$\bar h(y) = h(x) \Phi(x,y|\lambda)$. By assumption, the function $\wh T \to \C$ with
value $0$ on $\partial T$ and value $h(y)/\Phi(y|\lambda)$ at $y \in T$ is continuous. 
By uniform continuity
$$
\lim_{N \to \infty} \ep_N = 0\,, \quad \text{where}\quad 
\ep_N = \sup \{|h(y)/\Phi(y|\lambda)| : y \in T\,,\; |y| \ge N \}\,.
$$
Let $y \in T$ be such that $d(y,x) \ge N + |x|$. 
Then every $v \in T$ with $d(y,x)=d(y,v)$ satisfies $|v| \ge N$, 
so that 
$$
|h(v)| \le \ep_N\, |\Phi(v|\lambda)| 
= \ep_N\, |\Phi(x,v|\lambda)| \,\left|\frac{\Phi(v|\lambda)}{\Phi(x,v|\lambda)}\right|
$$
Applying Proposition \ref{pro:asy} once more, to both $\Phi(v|\lambda)$ and $\Phi(x,v|\lambda)\,$,
$$
\frac{\Phi(v|\lambda)}{\Phi(x,v|\lambda)} \sim \Fc(\lambda)^{|v| - d(v,x)} 
= \Fc(\lambda)^{|v| - d(y,x)} \quad \text{as }\; |y| \to \infty\,.
$$ 
Since $\Fc(\lambda)^{|v| - d(v,x)}$ is bounded in absolute value by 
$\max \{ |\Fc(\lambda)|^{|x|}, |\Fc(\lambda)|^{-|x|}\}$, we see that there is a finite
upper bound, say $M_x(\lambda)$, depending only on $x$ and $\lambda$, such that
$$
|h(v)| \le \ep_N\, |\Phi(x,v|\lambda)| \,M_x(\lambda) 
= \ep_N\, |\Phi(x,y|\lambda)|\,M_x(\lambda) \quad \text{whenever }\; d(v,x)=d(y,x)\,. 
$$
Consequently, also the absolute value of the average $\bar h(y)$ has the same upper bound. 
We get
$$
|h(x)| = \left|\frac{\bar h(y)}{\Phi(x,y|\lambda)}\right| \le 
\ep_N\, M_x(\lambda)\,. 
$$
Letting $N \to  \infty\,$, we conclude that $h(x)=0$, and this holds 
for any $x \in T$, as required.
\end{proof}

Theorem \ref{thm:laDir} tells us that for considering the boundary behaviour of a 
$\lambda$-polyharmonic 
function $f$ of order $n$, it first should be normalised by dividing by 
$\Phi^{(n-1)}(\cdot|\lambda)$. 

\begin{lem}\label{lem:uniquepoly} Let $f$ be polyharmonic of order $n$ and
 such that the $\lambda$-harmonic function $h= (\lambda \cdot I - P)^{n-1}f$ satisfies
$$
\lim_{x \to \xi} \frac{h(x)}{\Phi(x|\lambda)} = \gh(\xi) \quad \text{for all }\; 
\xi \in \partial T\,,
$$  
where $\gh \in \mathcal{C}(\partial T)$. Then 
$$
f(x) = \int_{\partial T} \gh(\xi)\, K(x,\xi|\lambda)\,
\hor^{(n-1)}(x,\xi|\lambda)\,d\xi \; + g\,, 
$$
where $g$ is $\lambda$-polyharmonic of order $n-1$.
\end{lem}
 
\begin{proof}
It follows from Theorems \ref{thm:launique} that 
$$
h(x) = h_{\gh}(x) = \int_{\partial T} \gh(\xi)\, K(x,\xi|\lambda)\, d\xi\,.
$$
Set 
$$
f_{\gh}(x) = \int_{\partial T} \gh(\xi)\, K(x,\xi|\lambda)\,
\hor^{(n-1)}(x,\xi|\lambda)\,d\xi\,.
$$
By Lemma \ref{lem:new}, 
$$
(\lambda \cdot I - P)^{n-1} f_{\gh} = h = (\lambda \cdot I - P)^{n-1} f\,.
$$
Therefore $g= f - f_{\gh}$ satisfies $(\lambda \cdot I - P)^{n-1}g = 0$.
\end{proof}

If in the above lemma, the natural normalisation $g\,\big/\,\Phi^{(n-2)}(\cdot|\lambda)$
has continuous boundary values, then $g\,\big/\,\Phi^{(n-1)}(\cdot|\lambda)$ tends to $0$
at the boundary of the tree by \eqref{eq:asy}. Thus, by Theorem \ref{thm:laDir}, 
$f\,\big/\,\Phi^{(n-1)}(\cdot|\lambda)$ has the same boundary limit $\gh$ as 
$(\lambda \cdot I - P)^{n-1}f\,\big/\,\Phi(\cdot|\lambda)$. 

We conclude that for considering an analogue of the classical Riquier problem,
with given boundary functions $\gh_0\,,\dots, \gh_{n-1}\,$, our solution
$f$ should be obtained step-wise: first,  
$(\lambda \cdot I - P)^{n-1}f\,\big/\,\Phi(\cdot|\lambda)$ should have boundary limit 
$\gh_{n-1}\,$, and we take $f_{n-1}=f_{\gh_{n-1}}$ according to Lemma \ref{lem:uniquepoly}.
Next, the function $f - f_{n-1}$ should be polyharmonic of order $n-1$, and
$(\lambda \cdot I - P)^{n-2}(f - f_{n-1})\,\big/\,\Phi(\cdot|\lambda)$ should have
boundary limit $\gh_{n-2}$. We then proceed recursively.  We clarify this by the
next definition.

\begin{dfn}\label{def:Riq}
 Let $\lambda \in \C \setminus [-\rho\,,\,\rho]$ and 
$\gh_0\,, \dots\,, \gh_{n-1} \in \mathcal{C}(\partial T)$. Then a solution of
the associated Riquier problem at infinity is a polyharmonic function
$$
f = f_0 + \dots + f_{n-1}
$$
of order $n$, where each $f_k$ is polyharmonic of order $k+1$ and
$$
\lim_{x \to \xi}  \frac{(\lambda\cdot I - P)^k f_k(x)}{\Phi(x|\la)} = \gh_k(\xi)  
\quad \text{for every }\; 
\xi \in \partial T \,.
$$
\end{dfn}


\begin{cor}\label{cor:Riq}
A solution of the Riquier problem as stated in Definition \ref{def:Riq}  is
given by  the functions 
$$
f_k(x) = \int_{\partial T} \gh_k(\xi)\, K(x,\xi|\lambda)
\,\hor_k(x,\xi|\lambda)\, d\xi\,.
$$
One also has
$$
\lim_{x \to \xi} \frac{f_0(x)+ \dots + f_k(x)}{\Phi_k(x|\la)} 
= \lim_{x \to \xi} \frac{f_k(x)}{\Phi_k(x|\la)} = \gh_k(\xi)  
\quad \text{for every }\; 
\xi \in \partial T.
$$
\end{cor}

As already outlined further above, the solution is \emph{not} unique. We can add to $f_k$ some suitable
$\lambda$-polyharmonic function of lower order: normalised by $\Phi_k(x|\lambda)$,
by \eqref{eq:asy}  the latter will tend to zero, as $|x| \to \infty$. What 
is unique is -- by Theorem \ref{thm:launique} -- 
the solution  $(\lambda\cdot I - P)^k f_k  = h_{\gh_k}\,$.

\end{document}